\documentclass[12pt,reqno]{amsart}
 
\numberwithin{equation}{section}

\def\({\left(}
\def\){\right)}

\def\NZQ{\mathbb}               
\def\NN{{\NZQ N }}
\def\QQ{{\NZQ Q}}
\def\ZZ{{\NZQ Z}}
\def\RR{{\NZQ R}}
\def\CC{{\NZQ C}}
\def\opn#1#2{\def#1{\operatorname{#2}}} 

\opn\chara{char}

\opn\gr{gr} \opn\rank{rank}
\let\:=\colon

\newtheorem{theorem}{Theorem}[section]
\newtheorem{lemma}[theorem]{Lemma}

\newtheorem{proposition}[theorem]{Proposition}
\newtheorem{conjecture}[theorem]{Conjecture}
\theoremstyle{definition}
\newtheorem{remark}[theorem]{Remark}

\newtheorem{question}[theorem]{Question}

\opn\ini{in} \opn\KRS{KRS}

\opn\krs{krs}
\opn\diag{diag}
\opn\DD{{\mathcal D}}
\opn\SS{{\mathcal S}}
\opn\MM{{\mathcal M}}
\opn\GL{GL}

\let\Bbb=\mathbb

\def\RR{{\Bbb R}}
\def\QQ{{\Bbb Q}}
\def\ZZ{{\Bbb Z}}
\def\NN{{\Bbb N}}
\def\ini{\operatorname{in}}

\opn\height{height}
\opn\length{length}
\opn\cl{cl}
\opn\Cl{Cl}
\opn\Grass{Grass}
\opn\sr{sr}
\opn\Hom{Hom}

\opn\Ker{Ker}
\opn\Im{Im}
\opn\supp{supp}

\opn\Spec{Spec}

\opn\SL{SL}
\opn\QF{QF}

\opn\GCD{gcd}

\def\qbinom#1#2{\left[\begin{matrix} #1 \\ #2\end{matrix}\right]}

\begin{document}

\title{Regular sequences of symmetric polynomials}
\author{ Aldo Conca, Christian Krattenthaler$^\dagger$,  Junzo Watanabe }
\address{Dipartimento di Matematica, Universit\`a di Genova, Via Dodecaneso, 35,
I-16146 Genova, Italy.
WWW: \tt http://www.dima.unige.it/\~{}conca/}

\address{Fakult\"at f\"ur Mathematik, Universit\"at Wien,
Nordbergstra{\ss}e~15, A-1090 Vienna, Austria.
WWW: \tt http://www.mat.univie.ac.at/\~{}kratt.}

\address{Department of Mathematics, Tokai University, 
Hiratsuka 259-1292, Japan}

\thanks{$^\dagger$Research partially supported 
by the Austrian Science Foundation FWF, grant S9607-N13,
in the framework of the National Research Network
``Analytic Combinatorics and Probabilistic Number Theory"}

\date{}

\begin{abstract}
A set of $n$ homogeneous polynomials in $n$ variables is a regular sequence if the associated polynomial system has only the obvious solution $(0,0,\dots,\break 0)$.  Denote by  $p_k(n)$  the power sum symmetric polynomial in $n$ variables $x_1^k+x_2^k+\dots+x_n^k$. The  interpretation of the $q$-analogue of the binomial coefficient as Hilbert function leads us to discover that $n$ consecutive power sums in $n$ variables  form a regular sequence. We consider then the following problem: describe the subsets $A\subset \NN^*$ of cardinality $n$ such that the set of polynomials $p_a(n)$ with $a\in A$ is a regular sequence.  We prove that a necessary condition is that $n!$ divides the product of the degrees of the elements of $A$. To find an easily verifiable sufficient condition turns out to be surprisingly difficult already for $n=3$.  Given positive integers $a<b<c$ with $\GCD(a,b,c)=1$, we conjecture that $p_a(3),p_b(3),p_c(3)$ is a regular sequence if and only if  $abc\equiv0$~$(\text{mod }6)$.  We provide evidence for the conjecture by proving it in several special instances.
\end{abstract}
\maketitle
 
\section{Introduction} 
For $d,n\in \NN$ let $P_{d,n}(q)$ be  the $q$-analogue of the binomial coefficient (also called 
Gau{\ss}ian polynomial) 
\begin{equation}
P_{d,n}(q):=\qbinom{n+d}{n}_q=\frac{[n+d]_q!}{[n]_q!\, [d]_q!}
\end{equation}
where 
$$[j]_q!=\prod_{i=1}^{j}\frac { 1-q^i} {1-q}=\prod_{i=1}^{j} (1+q+\dots+q^{i-1}) .$$
It is well known that  $P_{d,n}(q)\in \ZZ[q]$. Furthermore  for a finite field $F$,   $P_{d,n}(q)$ evaluated at $q=\vert F\vert$ gives exactly the number of $F$-subspaces of $F^{n+d}$ of dimension $n$, \cite[1.3.18]{S}.  It is also known that  the coefficient of  $q^k$  in  $P_{d,n}(q)$ is  the number of partitions of $k$ with at most $n$ parts  and parts  bounded by $d$, see \cite[1.3.19]{S}. This interpretation leads to the following construction. 
Consider the polynomial ring $R=\CC[x_1, x_2,\dots, x_n]$ with the usual action of the permutation group $G=S_n$. The monomial complete intersection $I=(x_1 ^{d+1}, x_2^{d+1},\dots, x_n^{d+1})$  is fixed by $G$, so that $G$ acts as well  on the quotient ring $A=R/I$.  The  homogeneous part $A_k$ of degree $k$ of $A$  has a $\CC$-basis  consisting  of the monomials in $x_1^{a_1}x_2^{a_2}\cdots x_n^{a_n}$  with $a_1+a_2+\dots +a_n=k$ and $a_i \leq d$ for all $i$.  For a partition $\lambda: \lambda_1\geq \lambda_2\geq  \dots \geq \lambda_n$  denote by $m_\lambda$ the monomial symmetric polynomial  
$$m_\lambda= x_1^{ \lambda_1} x_2^{ \lambda_2}\cdots x_n^{ \lambda_n}+\text{symmetric terms},$$
where ``symmetric terms" means the sum of all terms that have to be added to complete the monomial
$ x_1^{ \lambda_1} x_2^{ \lambda_2}\cdots x_n^{ \lambda_n}$ to a symmetric polynomial in $x_1,x_2,\dots,x_n$. 
 The invariant ring  $A^G$ in degree $k$  has  a $\CC$-basis   consisting of the elements $m_\lambda$ where $\lambda$ is a partition of $k$ in at most $n$ parts with parts smaller than $d+1$. In other words, $P_{d,n}(q)$ is the Hilbert series of $A^G$. 
 In characteristic $0$ the extraction of  the invariant submodule is an exact functor. So we have $A^G=R^G/I^G$,  where $R^G=\CC[e_1,e_2,\dots,e_n]$ and  $e_i$ is  the elementary symmetric polynomial of degree $i$.  Since the  $e_i$'s are algebraically independent, the Hilbert series of $R^G$ is $1\big/\prod_{i=1}^n (1-q^i)$. 
Since we know that $P_{d,n}(q)$ is the Hilbert series of $A^G$, we see that $A^G$ has the Hilbert series of a complete intersection in $R^G$ defined by elements of degree $d+1,d+2,\dots, d+n$.     
These  considerations suggest  that the  ideal $I^G=I \cap R^G$  might be generated by 
$$p_{d+1}, p_{d+2},\dots, p_{d+n}$$
where $p_i=x_1^i+x_2^i+ \cdots+ x_n^i$ is the power sum of degree $i$ and that they form a regular sequence. 
Since the inclusion $(p_{d+1}, p_{d+2},\dots, p_{d+n})\subseteq I^G$ is obvious,  to prove the equality it is enough to prove that      $p_{d+1}, p_{d+2},\dots, p_{d+n}$  form  a  regular sequence in $R^G$ or, which is the same,  in $R$.  We will see that this is indeed the case, see Proposition~\ref{Newton}. We give two (simple) proofs of Proposition~\ref{Newton};  one is based on  Newton's relations and the other on  Vandermonde's determinant.  Denote by $h_i$ the complete symmetric polynomial of degree $i$, that is, the sum of all the monomials of degree $i$ in $x_1,\dots,x_n$. More generally,  we are led to consider regular sequences of symmetric polynomials. 
In particular regular sequences of  power sums $p_i$  and regular sequences of complete symmetric polynomials $h_i$.  Even for $n=3$ these  questions  turn  out to be difficult. For  indices $a,b,c$ with $\GCD(a,b,c)=1$ we conjecture that  $p_a,p_b$ and $p_c$ form a regular sequence exactly when $6\mid abc$.   
 We are able to verify this conjecture in a few cases in which the property under investigation is translated into the non-vanishing of a rational number which appears as a coefficient in the relevant expressions or on the irreducibility over the 
rationals of certain polynomials obtained by elimination,  see  Theorem~\ref{betterthannothing}, Proposition~\ref{generalize} and Remark~\ref{Eisen}. 
Other interesting algebraic aspects of  ideals and algebras of  power sums are investigated by  Lascoux and Pragacz   \cite{LP}  and by Dvornicich and Zannier \cite{DZ}. 
We thank Riccardo Biagioli, Francesco Brenti,  Alain Lascoux and Michael Filaseta   for useful suggestions and comments.

\section{Regular sequences of symmetric polynomials}
Set  $R_n=\CC[x_1,x_2,\dots,x_n]$.  
 Following Macdonald, we denote by $e_k(n), p_k(n)$  and $h_k(n)$ the elementary symmetric polynomial,    the power sum  and the complete symmetric polynomial  of degree $k$  in $n$-variables, respectively. When $n$ is clear from the context or irrelevant  we  will simply denote them by   $e_k, p_k$  and $h_k$, respectively.    For a subset $A\subset \NN^*$ we set 
 $$p_A(n)=\{  p_a(n) : a\in A\}\quad \mbox{ and }\quad  h_A(n)=\{  h_a(n) : a\in A\}.$$

\begin{question} 
\label{mainq} 
For which subsets $A\subset \NN^*$ with $|A|=n$ is the set of polynomials $p_A(n)$ (respectively $h_A(n)$) a regular sequence  in $R_n$?  That is,  when is $(0,\dots,0)$ the  only common zero of the equations  $p_a(n)=0$  (respectively $h_a(n)=0$), $a\in A$? 
\end{question}

In the following we list several auxiliary observations.

\begin{lemma} \label{onlyalemma} 
Let $A\subset \NN^*$ with $|A|=n$. Set $d=\GCD(A)$ and $A'=\{ a/d : a\in A\}$.  Then    $p_A(n)$ is a   regular sequence   if and only if  $p_{A'}(n)$ is a   regular sequence.  
\end{lemma} 

\begin{proof} Note that if  $(z_1,z_2,\dots,z_n)$ is a common zero of the $p_a(n)$'s with $a\in A$ then 
$(z_1^d,z_2^d,\dots,z_n^d)$ is a common  zero of  the $p_a(n)$'s with $a\in A'$ and that if $(z_1,z_2,\dots,z_n)$ is a common zero of the $p_a(n)$'s with $a\in A'$ then 
$(z_1^{1/d},z_2^{1/d},\dots,z_n^{1/d})$ is a common  zero of  the $p_a(n)$'s with $a\in A$ where $z_i^{1/d}$ denotes any $d$-th root of $z_i$. 
\end{proof} 

\begin{remark} For a set $A$ of cardinality $n$ the condition that $p_A(n)$ is  a   regular sequence  can be rephrased in terms of a resultant: the resultant of the set must be non-zero.  More directly, one can express the same fact  by imposing the condition that  the ideal generated by $p_A(n)$  contains all the forms of degree  
$$\left(\sum_{a\in A} a\right)-n+1.$$ 
This boils down to the evaluation of the rank  of a  $\{0,1\}$-matrix of  large size. However, we  have not been able  to compute the resultant or to evaluate the rank of the matrix efficiently. 
\end{remark}

\begin{lemma}  \label{p2}  
 Let $A\subset \NN^*$ with $|A|=n$. Set $d=\GCD(A)$ and $A'=\{ a/d : a\in A\}$. Then we have:
\begin{itemize}
\item[(a)]  The equations $p_a(2)=0$, $a\in A$, have a non-trivial common zero in $\CC^2$  if and only if the elements in $A'$ are all odd.   
\item[(b)] The equations $h_a(2)=0$, $a\in A$, have a non-trivial common zero in $\CC^2$ if and only if $\GCD(\{a+1 : a\in A\})\ne1$. 
\end{itemize}
 \end{lemma} 
 
 \begin{proof} (a)  As in the proof of the lemma above, the polynomials  $p_a(2)$ with $a\in A$ have a non-trivial common zero if and only if 
 the polynomials $p_a(2)$ with $a\in A'$ have a non-trivial common zero. So we may assume that $A=A'$. If all the elements of $A$ are odd then $(1,-1)$ is a non-trivial common zero. Conversely, if there is a non-trivial common zero, then we may assume this zero to be $(x,1)$. In this case, we have $x^a=-1$ for all $a\in A$.   Since $\GCD(A)=1$ we can find a linear combination of the elements in 
 $A=\{a_1,a_2,\dots, a_n\}$ of the type $\sum  j_k a_k=1$ with $j_k\in \ZZ$. Hence we have $x^{j_k a_k}=(-1)^{j_k}$ and thus $x=(-1)^{\sum j_k}$.  Therefore we have either $x=1$,  which is impossible,  or $x=-1$ which implies that all the elements of  $A$ are odd. 
 
 (b)  Denote by (*) the condition: the equations $h_a(2)=0$ with $a\in A$ have a non-trivial common zero in $\CC^2$. So (*) holds  if and only if the equations  have a common zero of type $(c,1)$.  Obviously $c\neq  1$, 
so we can multiply each $h_a(2)$ by $c-1$.  We obtain that (*) holds if and only if  the equations $x^{a+1}-1=0$  with $a\in A$ have a common root $\neq 1$. In other words, $\GCD(\{x^{a+1}-1 : a\in A\})\neq x-1$. But  $\GCD(\{x^{a+1}-1 : a\in A\})=x^t-1$ 
 with $t=\GCD(\{a+1 : a\in A\})$, and we obtain the desired characterization. 
\end{proof}

This answers Question~\ref{mainq} for $n=2$. Namely, Lemma~\ref{p2} implies the following fact. 

\begin{lemma} 
\label{cosule}
Let $a,b\in \NN\setminus\{0\}$ and  $d=\GCD(a,b)$. Then we have:
\begin{itemize}
\item[(1)] $(p_a(2),p_b(2))$ is a   regular sequence if and only if 
$a/d$ or  $b/d$  is even. 
\item[(2)]  $(h_a(2),h_b(2))$ is a   regular sequence if and only if 
$\GCD(a+1,b+1)=1$.   
\end{itemize}
\end{lemma}

 Here is another fact.

 \begin{lemma} \label{rootsof1}   
 \begin{itemize} 
\item[(1)]  If $a\not\equiv 0$  $(\text{\em mod }c)$ and $u$ is a  primitive $c$-th root of unity then 
 $$(1,u,u^2,\dots,u^{c-1},0,\dots,0)\in \CC^n$$
  is a zero of  $p_a(n)$ in $\CC^n$ for all $n\geq c$. 
\item[(2)]  Let  $A\subset \NN^*$ with $|A|=n$ and  $1\leq c\leq n$. 
Set $\beta_c=|\{ a \in A :  a\equiv0$~$(\text{\em mod }c)\}|$. If  $\beta_c < \lfloor n/c \rfloor$   then  $p_A(n)$ is not a   regular sequence.  Here $\lfloor x \rfloor=\max\{ m\in \ZZ : m\leq x\}$ is the
standard floor function. 
\end{itemize} 
 \end{lemma} 
 
 \begin{proof} (1) Evaluating $p_a(n)$ at $(1,u,u^2,\dots,u^{c-1},0,\dots,0)$ yields $\sum_{i=1}^c u^{(i-1)a}$. Multiplying the sum by $u^a-1$, we obtain $u^{ca}-1$ which is $0$. Since $u$ is a primitive $c$-th root of unity and $a\not\equiv 0$~$(\text{mod }c)$, we have $u^a-1\neq 0$. It follows that $\sum_{i=1}^c u^{(i-1)a}=0$. 
 
 (2) Write $n=qc+r$ with $0\leq r<c$ so that $q=\lfloor n/c\rfloor$. Let $u$ be a primitive $c$-th root of unity. Let $(y_1,y_2,\dots,y_q)\in \CC^q$ and consider the element $\tilde y\in \CC^n$ obtained by concatenating $y_i(1,u,\dots,u^{c-1})$ for $i=1,2,\dots,q$ and by adding the zero vector of length $r$ at the end. By (1) we know that every such $\tilde y$ is a zero of the $p_a(n)$'s with $a\not\equiv 0$~$(\text{mod }c)$. If we also impose that $\tilde y$ is a zero of the $p_a(n)$'s with $a\in A$ and $a\equiv0$~$(\text{mod }c)$ we obtain $\beta_c$  homogeneous equations in $q$ variables. If $\beta_c<q$ then there exists a non-zero common solution to that system of equations. 
 \end{proof} 
 
We list another auxiliary result. 
 
 \begin{lemma} \label{hroots}  
 \begin{itemize} 
 \item[(1)]  Let $c>2$ and  $a+2\equiv 0$ or $1$~$(\text{\em mod }c)$. Let $u_1,u_2,u_3$ be distinct $c$-th roots of unity.   Then $p=(u_1,u_2,u_3,0,\dots, 0)\in \CC^n $ is a zero of  $h_a(n)$ in $\CC^n$ for all $n\geq 3$. 
\item[(2)]  Let  $A\subset \NN^*$ with $|A|=n$ and  $2<c$.   Assume that for all $a\in A$ one has $a+2\equiv 0$ or $1$~$(\text{\em mod }c)$. Then  $h_A(n)$ is not a   regular sequence.   \end{itemize} 
 \end{lemma} 
 
 \begin{proof} (1) We may assume $n=3$. To show that $h_a(3)$ evaluated at  $p=(u_1,u_2,u_3)$ is $0$ we first multiply $h_a(3)$ with $\Delta=(x_1-x_2)(x_1-x_3)(x_2-x_3)$. A simple calculation shows that 
$$h_a(3)\Delta=x_1^{a+2}(x_2-x_3)+x_2^{a+2}(x_3-x_1)+x_3^{a+2}(x_1-x_2).$$
Since, by assumption, we have $u_i^{a+2}=1$ or $u_i^{a+2}=u_i$ for all $i$,
this expression vanishes at  $p$.  
  
 (2) It follows immediately from (1) that $p$ is a common zero of the equations $h_a(n)$. So $h_A(n)$ is not a   regular sequence.  \end{proof} 
    
For the general case of $n$ homogeneous symmetric polynomials in $R_n$, there holds the following criterion.
    
 \begin{lemma}\label{facto}  Let $f_1,f_2,\dots,f_n$ be a   regular sequence of homogeneous symmetric polynomials in $R_n$.  Then $n!$ divides $(\deg f_1)(\deg f_2)\cdots (\deg f_n)$. \end{lemma}

\begin{proof}   For obvious reasons,  $f_1,f_2,\dots,f_n$  is also a   regular sequence in the ring of symmetric polynomials $\CC[e_1,e_2,\dots,e_n]$. Then the Hilbert series of 
$$\CC[e_1,e_2,\dots,e_n]/(f_1,f_2,\dots,f_n)$$ 
is $\prod_{i=1}^n (1-q^{\deg f_i})/ (1-q^i)$, and it must be a polynomial with integral coefficients. 
If we take the limit $q\to1$, we obtain $\big( \prod_{i=1}^n \deg f_i\big) \big/n! $, which must be an integer.   \end{proof} 
 
Next we prove that power sums, respectively complete symmetric polynomials, with consecutive indices form regular sequences.
 
 \begin{proposition} \label{Newton}
Let $A\subset \NN^*$  be a set of $n$ consecutive  elements.  Then both $p_A(n)$ and $h_A(n)$  are regular sequences  in $R_n$. 
 \end{proposition} 
 
 \begin{proof}   We present  two proofs for 
power sums and (sketch)  two proofs for complete symmetric polynomials. Let $a$ be the minimum of $A$.  We first prove the assertion for the power sums.  

Proof (1) for power sums: We argue by induction on $a$ and $n$.  If $n=1$ then  the assertion is obvious. If $a=1$, then any common zero of $p_1,p_2,\dots,p_n$ is also a common zero of $e_1,e_2,\dots,e_n$ because $p_1,p_2,\dots,p_n$ also generate the algebra  of symmetric polynomials.  But, obviously,  the only common zero of the elementary symmetric polynomials is $(0,\dots,0)$.  
 
Now assume $n>1$ and $a>1$.  For all $h\in \NN$ we have Newton's identity 
$$\sum_{k=0}^n (-1)^k e_{n-k} p_{k+h}=0,$$
with the convention that $p_0=n$ and $e_0=1$.
For $h=a-1$ we have 
$$e_n p_{a-1}=\sum_{k=1}^n (-1)^{k+1} e_{n-k} p_{k+a-1}.$$
If $z=(z_1,z_2,\dots,z_n)$ is a common zero  of $p_a,p_{a+1},\dots,p_{a+n-1}$ then it is also a  zero of $e_n p_{a-1}$. So $z$ is either a zero of $p_{a-1}$, and we conclude by induction on $a$ that $z=(0,\dots,0)$ or $z$ is  a zero of $e_n$. In the second case, one of the coordinates of $z$ is  $0$ and we conclude by induction on $n$ that $z=(0,\dots,0)$. 

Proof (2)  for power sums: Let 
$z=(z_1,z_2,\dots,z_n)\in \CC^n$ be a solution of the polynomial system associated to $p_A(n)$. We have to prove that $z=0$. 
Let $c$ be the cardinality of the set $\{ z_i : i=1,\dots,n\}$.  We may assume that 
$z_1,z_2,\dots,z_c$ 
are  the distinct values among $z_1,z_2,\dots,z_n$, 
and for $i=1,\dots,c$  set 
$m_i=\vert\{ j : z_j=z_i\}\vert$, 
so that $m_i>0$. Hence $p_k(z_1,\dots,z_n)=m_1z_1^k+\dots+m_cz_c^k$. 
By assumption, $p_{a+j}(z_1,\dots,z_n)=0$ for $j=0,\dots, n-1$.  Let  $V$ be the Vandermonde matrix 
 $(z_j^i)$ with $i=0,\dots,c-1$ and $j=1,\dots,c$. By construction,   $\det V\neq 0$ and the column vector $(m_1z_1^a, m_2z_2^a, \dots,  m_c z_c^a)$ is in the kernel of $V$.  It follows that $(m_1z_1^a, m_2z_2^a, \dots,  m_cz_c^a)=0$, that is, $c=1$ and $z_1=0$. This shows that $z=0$. 

Proof (1)  works  also for complete symmetric polynomials by  replacing Newton's identity with the corresponding identity for the complete symmetric polynomials.  For a  second proof for complete symmetric polynomials, one observes  that the initial ideal  with respect to any term order with $x_1>x_2>\dots>x_n$  of the ideal generated by  $h_a(n), h_{a+1}(n), \dots, h_{a+n-1}(n)$ contains (and hence  is equal to) $(x_1^a,x_2^{a+1}, \dots, x_n^{a+n-1})$. 
\end{proof} 

 The  above results (and  many computer calculations)  suggest the following conjecture.
 
 \begin{conjecture} \label{con3} 
  Let $A=\{a,b,c\}$, where $a,b,c$ are positive integers with $a<b<c$ and $\GCD(A)=1$.  Then  $p_A(3)$ is a regular sequence in $R_3$  if and only if $abc\equiv0$~$(\text{\em mod }6)$. 
 \end{conjecture} 
 
The ``only if" part has been proved in  Lemma~\ref{facto}. 
In direction of the ``if" part, we are able to offer the following partial result.

 \begin{theorem} \label{betterthannothing} 
Conjecture~{\em\ref{con3}} is true if $A$ either contains $1$ and $n$ with $2\leq n\leq 7$, or if it contains $2$ and $3$.  \end{theorem}

\begin{proof}   For simplicity, let us denote by $p_k$ and $e_k$ the corresponding symmetric polynomials in $3$ variables. In the following considerations we will use  the basic fact that $e_1,e_2,e_3$ are algebraically independent generators for the algebra of symmetric polynomials in $x_1,x_2,x_3$.
 
Assume first that $A=\{1,n,m\}$ with  $2\leq n\leq 7$,  $m\neq n$, and $6\mid nm$. 
 Formulas expressing the power sum $p_{h}$ in terms of the elementary symmetric polynomials  are well-known, see \cite[Ex.~20, p.~33]{M}.  We will  use the  fact that  every monomial $e_1^{\beta_1}e_2^{\beta_2}e_3^{\beta_3}$ with $\beta_1+2\beta_2+3\beta_3=h$ appears in the expression of $p_h$ with a non-zero coefficient.   The cases  $n=2,3,4,5,7$ are easy since in those cases $p_n$ is a monomial in $e_2$ and $e_3$ mod $(e_1)$.  For instance, $p_4=ue_2^2$ mod $(e_1)$ and $p_5=ue_2e_3$ mod $(e_1)$, where $u$ stands for a non-zero integer.  This is enough to show that 
$$
\sqrt{(p_1,p_n)}=\begin{cases}
\sqrt{(e_1,e_2)} &  \mbox{  if } n=2,4, \\
\sqrt{(e_1,e_3)} &  \mbox{  if } n=3,\\ 
\sqrt{(e_1,e_2e_3)} &  \mbox{  if } n=5,7,
\end{cases}
$$
where $\sqrt{I}$ denotes the radical of the ideal $I$. In the cases $n=2$ or $n=4$,  we have  $m=3v$  , hence $p_m=ue_3^v$ mod $(e_1,e_2)$ for some non-zero integer $u$. This implies that $\sqrt{(p_1,p_n,p_m)}=\sqrt{(e_1,e_2,e_3)}$. One concludes in a similar manner in the cases $n=3$ and $n=5,7$.

 The proof  for $n=6$ is more complicated since $p_6$ is not a monomial mod $(e_1)$. Indeed,  $p_6=-2e_2^3+3e_3^2$ mod $(e_1)$.  However,  reducing $p_m$ mod $(e_1)$ and $p_6$, that is, replacing $e_1$ with $0$  and  $e_3^2$ with $(2/3)e_2^3$ we obtain 
$$p_m=
\begin{cases}
\hphantom{e_1^{h-1}}a_m e_2^h \mod (p_1, p_6) &  \mbox{ if } m=2h, \\ 
a_m e_2^{h-1}e_3 \mod (p_1, p_6)&  \mbox{ if } m=2h+1 ,\\
\end{cases}
$$
where $a_m$ is an integer. The assertion that we have to prove   is   equivalent to the non-vanishing of  coefficient $a_m$.  The integer $a_m$ can be computed using the formula  expressing $p_m$ in terms of the $e_i$'s, see \cite[Ex.~20, p.~33]{M}. Explicitly, we have
$$a_m= 
\begin{cases}
\displaystyle
\phantom{-}m
\sum_{b=0} ^{\lfloor h/3 \rfloor}\frac{(-1)^{h-b}}{h-b} 
\binom{h-b}{2b} (2/3)^b & \mbox{ if } m=2h , \\  \\
\displaystyle
-m
\sum_{b=0}  
^{\lfloor h/3 \rfloor}\frac{(-1)^{h-b}}{ h-b} \binom{h-b}{2b+1} (2/3)^b & \mbox{ if } m=2h+1 . 
\end{cases}
$$
To show that $a_m\neq 0$ for $m\neq 1$ and $6$,  we consider 
$$f_m(x)= 
\begin{cases}
\displaystyle
\phantom{-}m
\sum_{b=0} ^{\lfloor h/3 \rfloor} \frac{(-1)^{h-b}}{h-b} 
\binom{h-b}{2b} x^b & \mbox{ if } m=2h ,\\ \\
\displaystyle
-m
\sum_{b=0}  ^{\lfloor h/3 \rfloor}\frac{(-1)^{h-b}}{ h-b} \binom{h-b}{2b+1} x^b & \mbox{ if } m=2h+1 . 
\end{cases}
$$
Since
$$\frac {2h} {h-b}\binom {h-b}{2b}=2\binom {h-b}{2b}-\binom {h-b-1}{2b-1}$$
and
$$\frac {2h+1} {h-b}\binom {h-b}{2b+1}=2\binom {h-b}{2b+1}-\binom {h-b-1}{2b},$$
the polynomials $f_m(x)$ are in $\ZZ[x]$. We have to show that $2/3$ is not a root of $f_m(x)$ for $m\neq 1,6$.  If $m<8$,  then  $f_m(x)$ is a non-zero constant. So we may assume $m\geq 8$. If $m$ is odd, then the coefficient of the term of degree $0$ in $f_m(x)$ is odd. If $m$ is even and $m\not\equiv 0$~$(\text{mod }6)$ and $m\not\equiv 10$~$(\text{mod }18)$ then   the leading coefficient of  of $f_m(x)$ is not divisible by $3$.  This is enough to conclude that $2/3$ is not a root of $f_m(x)$ in this case. 
If $m\equiv0$~$(\text{mod }6)$ or $m\equiv10$~$(\text{mod }18)$ one needs a more sophisticated analysis of the $3$-adic valuation of the other  coefficients of $f_m(x)$. The argument  is given in the appendix.  

Finally, assume that $A$ contains $2$ and $3$, say $A=\{2,3,d\}$ for some $d>3$. Since $p_1,p_2,p_3$ generates the algebra of symmetric polynomials in $x_1,x_2,x_3$, we have $p_d=c_dp_1^d$ mod $(p_2,p_3)$ for a uniquely determined rational number $c_d$. The statement we have to prove is equivalent to the non-vanishing of $c_d$. 
Since $e_2=\frac 12 p_1^2$  and $e_3=\frac 16p_1^3$ mod $(p_2,p_3)$, from Newton's equation $p_d=e_1p_{d-1}-e_2p_{d-2}+e_3p_{d-3}$  we obtain that 
$$c_d=c_{d-1}-\frac 12c_{d-2}+\frac 16c_{d-3} \mbox{ for } d>3, \mbox{ with } c_1=1, c_2=0, c_3=0.$$
Solving the linear recurrence, see \cite[4.1.1]{S}, we get that 
$$c_d=\alpha^d+\beta^d+\bar{\beta}^d,$$
where $\alpha\in \RR$, $0<\alpha<1$, and $\beta,\bar{\beta}\in \CC$ are the roots of the polynomial 
$$x^3-x^2+\frac 12x-\frac 16=0.$$
We  show that $c_d>0$ for $d>3$. To this end,  it is enough to show that $\alpha^d>2|\beta|^d$ for $d>3$, equivalently, $(\alpha/|\beta|)^d>2$. Hence it is enough to prove the statement for $d=4$. With the help of a computer algebra program, we find that $(\alpha/|\beta|)^4=2.17\dots$.
\end{proof} 

In order to generalize part of Theorem~\ref{betterthannothing} we need the following lemma.  

\begin{lemma}
\label{modulo1}
Let $A=\{a,b,c\}$, where $\GCD(A)=1$ and $abc\equiv0$~$\text{\em(mod }3)$. Then the zero-set of the polynomial system associated to $p_A(3)$ intersects  $\{ (z_1,z_2,z_3)\in \CC^3 : |z_1|=|z_2|=|z_3| \}$ only in $(0,0,0)$. 
\end{lemma}
\begin{proof} By contradiction, assume $(z_1,z_2,z_3)\in \CC^3$ is a solution of the polynomial system associated to $p_A(3)$ and $|z_1|=|z_2|=|z_3|\neq 0$. Dividing by $z_3$, we may assume that $z_3=1$ and $|z_1|=|z_2|=1$. Note that the only complex numbers $w_1,w_2$ satisfying  $|w_1|=|w_2|=1$ and $w_1+w_2+1=0$ are the two primitive third roots of unity. 
Hence $z_1^a,z_1^b$ and $z_1^c$ are primitive third roots of $1$. Since $\GCD(a,b,c)=1$, there exist integers $\alpha,\beta, \gamma$  such that $a\alpha+b\beta+c\gamma=1$. It follows that $z_1=z_1^{a\alpha+b\beta+c\gamma}=(z_1^a)^\alpha(z_1^b)^\beta(z_1^c)^\gamma$ and hence $z_1$ itself is a third root of $1$. But one among $a,b,c$, say $a$, is divisible by $3$. Hence $z_1^a=1$ which is a contradiction. 
\end{proof}

We can now state, and prove, the announced partial generalization of
Theorem~\ref{betterthannothing}.

\begin{proposition} 
 \label{generalize} 
 Conjecture~{\em\ref{con3}} holds  if  $A$ contains $a$ and $at$ with 
$t\in \{2,3,4,\break 5,7\}$. 
\end{proposition}
\begin{proof} 
Let $\rho$ be a primitive third root of $1$. We claim that for $t\in \{2,3,4,5,7\}$ the zero-set of $p_1(3)$ and $p_t(3)$ consists (up to multiples and permutations) of at most $(1,\rho,\rho^2)$ (if $t\not\equiv 0$~$\text{mod }3)$) and $(1,-1,0)$ (if $t$ is odd).  
The assertion follows from the fact that, for these values of $t$, $p_t(3)$ is a monomial in $e_2(3)$ and $e_3(3)$~$\text{mod }(e_1(3))$.  
From Lemma~\ref{cosule} it follows that every non-zero point $z=(z_1,z_2,z_3)\in \CC^3$ in the zero set of $p_A(3)$  satisfies $z_1z_2z_3\neq 0$. Hence, according to the claim above, $(z_1^a,z_2^a,z_3^a)$ equals $(1,\rho,\rho^2)$ (up to multiples and permutations). Hence $|z_1|=|z_2|=|z_3|$ and this contradicts Lemma~\ref{modulo1}. 
\end{proof}
 
\begin{remark} 
\label{Eisen}
For $b\in \NN$, $b>1$, set $f_b(x)=1+x^b+(-1)^b(x+1)^b$. Note that $f_b(x)=p_b(x,1,-1-x)$.  
Let $A=\{1,a,b\}$ with 
$ab\equiv0$~(mod~6).  Conjecture~\ref{con3}  for $A$ is equivalent to  $\gcd(f_a(x), f_b(x))=1$.  We expect $f_b(x)$ to be irreducible in $\QQ[x]$  up to the factor $x$ and cyclotomic factors $1+x$ or $1+x+x^2$ which are present or not depending on 
$b$~mod~6.  In particular,  we expect $f_b(x)$ to be irreducible $\QQ[x]$ if 
$b\equiv0$~(mod~6). A simple computation shows that Eisenstein's criterion  applies to   
$f_b(x+1)$  with respect to $p=3$ for $b$ of the form $3^u(3^v+1)$ with $u>0$ and $v\geq 0$. These considerations imply that, if $b$ is of that form,   then $f_b$ is irreducible (over $\QQ$).  Hence,  Conjecture~\ref{con3} holds for $A=\{1,a,b\}$ where $a<b$ and $b=3^u(3^v+1)$ with $u>0$ and $v\geq 0$. \end{remark}

For $n>3$  the condition $\big(\prod_{i\in A} i\big)\equiv0$~$(\text{mod }n!)$ does not imply that $p_A(n)$ is a   regular sequence.  For $n=4$, computer experiments suggest the following conjecture.

\begin{conjecture} \label{C3}
Let $A\subset \NN^*$ with 
$\vert A\vert=4$, say $A=\{a_1,a_2,a_3,a_4\}$, and assume $\GCD(A)=1$. Then $p_A(4)$  is a   regular sequence if and only if  $A$ satisfies the following conditions: 

\begin{itemize} 
\item[(1)]  At least two of the $a_i$'s are even, at least one is  a multiple of $3$,  and at least one is a  multiple of $4$. 
\item [(2)]  If $E$  is the set of the even elements in A and $d=\GCD(E)$ then the set   $\{ a/d : a \in E\}$  contains an even number. 
\item[(3)] $A$ does not contain a subset of the form  $\{d,2d,5d\}$.
\end{itemize} 
\end{conjecture} 

\begin{remark}
(a) Condition (1) is obviously stronger than $4!$ divides $a_1a_2a_3a_4$. 
For instance, the set $\{1,3,5,8\}$ does not satisfy (1) and the product of its elements is divisible by  $4!$. 

\smallskip
(b) The conditions (2) and (3)  are  independent.
For example, the set $\{1,3,4,12\}$   satisfies (1) but not (2) and  $\{1,2,5,12\}$   satisfies (1) and (2) but not (3). 

\smallskip
(c)
We can prove  that  conditions  (1), (2) and (3) are necessary. Indeed,  assume $p_A(4)$ is a   regular sequence. Then $(1)$ holds by Lemma~\ref{rootsof1}.  
To get (2), consider the point $P=(x,-x,y,-y)\in \CC^4$. Obviously $P$ is a solution of the equation $p_k(4)=0$ with $k$ odd.    If  $k$ is even, then  $p_k(4)$ evaluated at $P$  is $2p_k(2)$ evaluated at $(x,y)$.  Hence we see that the only common root of  $p_k(2)$ with   $k\in E$  is $(0,0)$. So, by Lemma~\ref{p2}, at least one element of $\{ a/d : a \in E\}$ is even.  To show that (3) is necessary, note that, by degree reasons, $p_5(4)$ is in the ideal 
$(p_1(4), p_2(4))$. Replacing $x_i$ by $x_i^d$, we see that $p_{5d}(4)$ is in the ideal 
$(p_d(4), p_{2d}(4))$. So (3) follows. 
\end{remark}

For complete symmetric polynomials in three variables  we formulate the following
conjecture.

\begin{conjecture} Let $A=\{a,b,c\}$ with $a<b<c$.  Then $h_A(3)$ is a   regular sequence if and only if  the following conditions are satisfied: 
\begin{itemize}
\item[(1)]   $abc\equiv0$~$(\text{\em mod }6)$.
\item[(2)] $\GCD(a+1,b+1,c+1)=1$.
\item[(3)] For all $t\in \NN$ with $t>2$  there exist $d\in A$ such that $d+2\not\equiv 0,1$~$(\text{\em mod }t)$.
\end{itemize} 
\end{conjecture} 
 
Again, the ``only if" part follows from the considerations above. For instance, if (2) does not hold then the three polynomials have a common non-zero solution of the form $(x,1,0)$.

\appendix

\section*{Appendix: Non-vanishing of the coefficient $a_m$} \label{ak} 
 
\global\def\theequation{A.\arabic{equation}}

We want to prove that
\begin{equation} \label{eq:1}
\sum _{b=0} ^{\lfloor{h/3}\rfloor}\frac{(-1)^{h-b}} {h-b}\binom {h-b} {2b}\(\frac {2}
{3}\)^{b}
\end{equation}
is non-zero except for $h=3$.

We may assume from now on that $h>3$.
The idea is a $3$-adic analysis of the summands.
All of them are rational numbers.
If $h\not\equiv 3,5$~(mod~9), then we shall show that the $3$-adic
valuation of the last summand (the summand for $b=\lfloor h/3\rfloor$
is smaller than the $3$-adic valuations of all other summands.
In this situation, the sum cannot be zero.
Similarly, if $h\equiv 3,5$~(mod~9), we shall show that the $3$-adic
valuations of the summands for $b\le\lfloor\frac {h} {3}\rfloor-2$ are all larger
than the $3$-adic valuation of the sum of the two summands for
$b=\lfloor\frac {h} {3}\rfloor$ and $b=\lfloor\frac {h} {3}\rfloor-1$. Here, summands with $b\le\lfloor\frac {h} {3}\rfloor-2$ do indeed exist (since we assumed
that $h>3$ which, together with $h\equiv 3,5$~(mod~9),
implies that $h$ must be at least
$3+9=12$), and therefore the sum cannot be zero.

We write $v_3(\frac {r} {s})$ for the $3$-adic valuation of the
rational number $\frac {r} {s}$ which, by definition, is $3^{a-b}$,
where $3^a$ is the largest power of $3$ dividing $r$, and
where $3^b$ is the largest power of $3$ dividing $s$.

We shall use the following (well-known and easy to prove) fact:
the $3$-adic valuation of a binomial coefficient $\binom mn$,
$v_3\(\binom {m+n}n\)$, is equal to the number of carries which occur
during the addition of the numbers $m$ and $n$ in ternary notation.
From now on, whenever we speak of ``carries during addition of two
numbers", we always mean the addition of these two numbers when
written in ternary notation.

\medskip
{\it Case 1}: $h\equiv0$~(mod 3). Let $h=3k$, $k>1$. In \eqref{eq:1} replace
$b$ by $k-b$ to obtain the sum
\begin{equation} \label{eq:2}
\sum _{b=0} ^{k}\frac {(-1)^{b}} {2k+b}\binom {2k+b} {3b}\(\frac {2}
{3}\)^{k-b}.
\end{equation}
Let first $k\not\equiv1$~(mod 3) (i.e., $h\not\equiv3$~(mod~9)).

For the summand in \eqref{eq:2} for $b=0$, we have
$$v_3\(\frac {1} {2k}\(\frac 2
{3}\)^{k}\)\le -k.$$
We claim that we have
\begin{equation} \label{eq:3}
v_3\(\frac {1} {2k+b}\binom {2k+b} {3b}\(\frac {2}
{3}\)^{k-b}\)>-k
\end{equation}
for all $b\ge1$.

To prove the claim we assume that $v_3(2k+b)=e$ and that
$3^{s-1}\le 3b<3^s$. Clearly $e\ge0$ and $s\ge2$. Using these
conventions, we have
\begin{multline} \label{eq:4}
v_3\(\frac {1} {2k+b}\binom {2k+b} {3b}\(\frac {2}
{3}\)^{k-b}\)\\
=b-k-e+\#(\text {carries during addition of $3b$ and
$2k-2b$}).
\end{multline}
Since $b\ge1$, the inequality \eqref{eq:3} holds for $e=0$.
From now on we assume that $e>0$.

If $s>e>0$, then from \eqref{eq:4} we get
\begin{align*} v_3\(\frac {1} {2k+b}\binom {2k+b} {3b}\(\frac {2}
{3}\)^{k-b}\)&\ge3^{s-2}-k-e\\
&\ge-k+3^{e-1}-e\\
&\ge-k.
\end{align*}
This proves \eqref{eq:3} for $e>1$, since in this case we have actually
$3^{e-1}>e$, and thus $>-k$ in the last line of the inequality chain.
It also proves \eqref{eq:3} for $e=1$ and $s>2$, since in this case we have
$3^{s-2}>3^{e-1}$, and thus we have $>-k+3^{e-1}-e$ in the inequality
chain. The only case which is left open is $e=1$ and $s=2$.

Let now $s\le e$. Let us visualize the numbers $2k+b$ and $3b$ in
ternary notation,
\begin{align*} (2k+b)_3&=\dots\underbrace {0\dots\dots0} _e,\\
(3b)_3&=\hphantom{0\dots\dots}\underbrace{\dots 0}_s.
\end{align*}
When we add $2k-2b$ to $3b$, then we get $2k+b$. Since $(2k+b)_3$ has
0's as $s$-th, $(s+1)$-st, \dots, $e$-th digit (from the right),
we must have at least $e-s+1$ carries when adding $2k-2b$ and $3b$.
From \eqref{eq:4} we then have
\begin{align*} v_3\(\frac {1} {2k+b}\binom {2k+b} {3b}\(\frac {2}
{3}\)^{k-b}\)&\ge3^{s-2}-k-e+(e-s+1)\\
&\ge-k+3^{s-2}-s+1\\
&\ge-k.
\end{align*}
The last inequality is in fact strict if $s>2$, proving the claim \eqref{eq:3}
in this case.

In summary, except for $s=2$ and $e>0$, the claim \eqref{eq:3} is proved.
However, if $s=2$, then $b=1$ or $b=2$. If $b=2$, then we have
$b>3^{s-2}$. If we use this in the first line of the inequality
chains, then both of them become strict inequalities.
Therefore the only case left is $b=1$ and $e>0$. However, if $b=1$,
then $e=v_3(2k+b)=v_3(2k+1)=0$ since we assumed that
$k\not\equiv1$~(mod~3). This is absurd, and hence the claim is established completely.

\smallskip
Now we address the (more complicated) case that $k\equiv1$~(mod~3)
(i.e., $h\equiv3$~(mod~9)). In this case we combine the summands in \eqref{eq:2} for $b=0$ and $b=1$:
$$\frac {1} {2k}\(\frac {2} {3}\)^k-
\frac {1} {2k+1}\binom {2k+1}3\(\frac {2} {3}\)^{k-1}=
-\frac {2^{k-1}} {3^k}\frac {(k-1)(2k^2+k+1)} {k}.
$$
The $3$-adic valuation of this expression is
$$v_3(k-1)-k.$$
Let us write $f=v_3(k-1)$, and, as before, $v_3(2k+b)=e$ and
$3^{s-1}\le 3b<3^s$. 
Since $k\equiv1$~(mod), we know that $f\ge1$.
We claim that for $b\ge2$ we have
\begin{equation} \label{eq:5}
v_3\(\frac {1} {2k+b}\binom {2k+b} {3b}\(\frac {2}
{3}\)^{k-b}\)>-k+f.
\end{equation}
Since the notation is as before, we may again use Equation~\eqref{eq:4} for the computation of the $3$-adic valuation on the left-hand
side.

For convenience, we visualize the involved numbers, \begin{align*} (2k)_3&=\dots\underbrace{0\dots\dots\dots02}_f\\
(2k+b)_3&=\dots\dots\dots\underbrace{0\dots00}_e\\
(b)_3&=\hphantom{\dots\dots}\underbrace{\dots\underbrace{2\dots21}_
{\min\{e,f\}}}
_{s-1}\\
(3b)_3&=\hskip-7pt\hphantom{\dots\dots}
\underbrace{\dots\underbrace{22\kern1pt.\kern1pt.\kern1pt210}_{
\min\{e,f\}+1}}_{s}
\end{align*}
For later use, we note that we must have
\begin{equation} \label{eq:6}
b\ge 3^{\min\{e,f\}}-2,
\end{equation}
and if $s-1>e$ even
\begin{equation} \label{eq:7}
b\ge 3^{s-2}+3^{\min\{e,f\}}-2.
\end{equation}
Another general observation is that, if $e>0$, then
the number of carries when adding $3b$ and $2k-2b$ must be at least
\begin{equation} \label{eq:8}
e+\chi(f\ge e)\cdot\max\{0,f-s+\chi(s\not=e+1)\},
\end{equation}
because there must be carries when adding up the second, third,\dots,
$(e+1)$-st digits (from the right), and because $(2k+b)_3$ has 0's as
$s$-th, $(s+1)$-st, \dots, $f$-th digit.
Here, $\chi(\mathcal A)=1$ if $\mathcal A$ is true and $\chi(\mathcal A)=0$
otherwise. The truth value $\chi(s\not=e+1)$ occurs because if $s=e+1$ we
cannot count the carry at the $(s=e+1)$-st digit twice.

\smallskip
(a) $e\ge f\ge1$. From \eqref{eq:4}, \eqref{eq:6}, and \eqref{eq:8}, we get
\begin{align*} v_3\(\frac {1} {2k+b}\binom {2k+b} {3b}\(\frac {2}
{3}\)^{k-b}\)&\ge (3^f-2)-k-e+e\\
&\ge -k+f.
\end{align*}
Moreover, if
$f>1$ then the last inequality is strict, proving the claim \eqref{eq:5} in
this case. The only case which remains open is $f=1$.
If $s>2$, we could use \eqref{eq:7} instead of \eqref{eq:6}, in which case the
claim would also follow. Thus, we are left with considering
the case $s=2$ and $f=1$. In that case, we would have $b=1$,
a contradiction.

(b) $f\ge s=e+1\ge2$. From \eqref{eq:4}, \eqref{eq:6}, and \eqref{eq:8}, we get
\begin{align*} v_3\(\frac {1} {2k+b}\binom {2k+b} {3b}\(\frac {2}
{3}\)^{k-b}\)&\ge 3^e-2-k-e+(e+f-s)\\
&\ge-k+f+3^{s-1}-s-2\\
&>-k+f,
\end{align*}
as long as $s>2$. On the other hand, if $s=2$ and, hence, $e=1$, then we would have $b=1$, a contradiction.

\smallskip
(c) $f\ge s> e+1\ge2$.
From \eqref{eq:4}, \eqref{eq:7}, and \eqref{eq:8}, we get
\begin{align*} v_3\(\frac {1} {2k+b}\binom {2k+b} {3b}\(\frac {2}
{3}\)^{k-b}\)&\ge 3^{s-2}+3^e-2-k-e+(e+f-s+1)\\
&\ge3^{s-2}+3^e-1-k+f-s\\
&>(s-1)+1-k+f-s\\
&>-k+f,
\end{align*}
since $s>2$ and $e\ge1$.

\smallskip
(d) $s>f >e\ge1$. From \eqref{eq:4}, \eqref{eq:7}, and \eqref{eq:8}, we get
\begin{align*} v_3\(\frac {1} {2k+b}\binom {2k+b} {3b}\(\frac {2}
{3}\)^{k-b}\)&\ge 3^{s-2}+3^e-2-k-e+e\\
&\ge 3^{f-1}+3^e-2-k\\
&>-k+f,
\end{align*}
since $f>1$ and $e\ge1$.

\smallskip
(e) $f >e=0$. Our previous visualization of the involved numbers then
collapses to
\begin{align*} (2k)_3&=\dots\underbrace{0\dots\dots\dots02}_f\\
(2k+b)_3&=\dots\dots\dots\dots\dots.\\
(b)_3&=\hskip4pt\hphantom{\dots\dots}\underbrace{\dots\dots\dots}
_{s-1}\\
(3b)_3&=\hskip-8pt\hphantom{\dots\dots.}
\underbrace{\dots\dots\dots0}_{s}
\end{align*}
Instead of \eqref{eq:6} or \eqref{eq:7}, we have now $b\ge3^{s-2}$.
The number of carries when adding $3b$ and $2k-2b$ must still be at least
$\max\{0,f-s+1\}$. Hence, from \eqref{eq:4} we get
\begin{equation} \label{eq:9}
v_3\(\frac {1} {2k+b}\binom {2k+b} {3b}\(\frac {2}
{3}\)^{k-b}\)=3^{s-2}-k+\max\{0,f-s+1\}.
\end{equation}

If $f< s$, then we obtain from \eqref{eq:9} that
\begin{align*} v_3\(\frac {1} {2k+b}\binom {2k+b} {3b}\(\frac {2}
{3}\)^{k-b}\)&\ge 3^{f-1}-k\\
&\ge-k+f,
\end{align*}
since $f\ge1$.
This inequality chain is in fact strict as long as $s>f+1$
or $f>1$. Thus, the only case which is open is $f=1$ and $s=2$.
In that case, we must necessarily have $b=2$. So, instead
of $b\ge3^{s-2}$ to obtain \eqref{eq:9}, we could have used $b>1=3^{s-2}$,
again leading to a strict inequality.

If $f\ge s$, then we obtain from \eqref{eq:9} that
\begin{align*} v_3\(\frac {1} {2k+b}\binom {2k+b} {3b}\(\frac {2}
{3}\)^{k-b}\)&\ge 3^{s-2}-k+f-s+1\\
&\ge-k+f,
\end{align*}
since $s\ge2$. This inequality chain is in fact strict whenever $s\ge3$. If $s=2$ then the same argument
as in the previous paragraph leads also to a strict inequality.

\smallskip
This completes the verification of the claim \eqref{eq:5}.

\medskip
{\it Case 2}: $h\equiv1$~(mod 3). Let $h=3k+1$, $k\ge1$. In \eqref{eq:1} replace
$b$ by $k-b$ to obtain the sum
\begin{equation} \label{eq:10}
\sum _{b=0} ^{k}\frac {(-1)^{b+1}} {2k+b+1}\binom {2k+b+1} {3b+1}\(\frac {2}
{3}\)^{k-b}.
\end{equation}

For the summand in \eqref{eq:10} for $b=0$, we have
$$v_3\(-\frac {1} {2k+1}(2k+1)\(\frac 2
{3}\)^{k}\)= -k.$$
We claim that we have
$$v_3\(\frac {1} {2k+b+1}\binom {2k+b+1} {3b+1}\(\frac {2}
{3}\)^{k-b}\)>-k$$
for all $b\ge1$. To see this we argue as in Case~1. Everything works in complete
analogy. However, what makes things less complicated here is the fact
that the first digit (from the right) of $3b+1$ is a 1.
Therefore there is an additional carry when adding $3b+1$ and $2k-2b$
(in comparison to the addition of $3b$ and $2k-2b$ in Case~1; namely
when adding the first digits), and this implies that the complications
that we had in Case~1 when $k\equiv1$~(mod~3) do not arise here.

\medskip
{\it Case 3}: $h\equiv2$~(mod 3). Let $h=3k+2$, $k\ge1$. In \eqref{eq:1} replace
$b$ by $k-b$ to obtain the sum
\begin{equation} \label{eq:11}
\sum _{b=0} ^{k}\frac {(-1)^{b}} {2k+b+2}\binom {2k+b+2} {3b+2}\(\frac {2}
{3}\)^{k-b}.
\end{equation}

For the summand in \eqref{eq:11} for $b=0$, we have
$$v_3\(\frac {1} {2k+2}\frac {(2k+2)(2k+1)}2\(\frac 2
{3}\)^{k}\)= -k+v_3(2k+1).$$
Let us write $f=v_3(2k+1)$. Thus,
\begin{equation} \label{eq:12}
v_3\(\frac {1} {2k+2}\frac {(2k+2)(2k+1)}2\(\frac 2
{3}\)^{k}\)= -k+f.
\end{equation}
Similarly to earlier, we assume that $e=v_3(2k+b+2)$ and that
$3^{s-1}\le 3b+3<3^s$. Then the analogue of \eqref{eq:4} is
\begin{multline} \label{eq:13}
v_3\(\frac {1} {2k+b+2}\binom {2k+b+2} {3b+2}\(\frac {2}
{3}\)^{k-b}\)\\
=b-k-e+\#(\text {carries during addition of $3b+2$ and
$2k-2b$}).
\end{multline}

For convenience, we visualize the involved numbers, \begin{align*} (2k+1)_3&=\dots\underbrace{0\dots\dots\dots00}_f\\
(2k+b+2)_3&=\dots\dots\dots\underbrace{0\dots00}_e\\
(b+1)_3&=\hphantom{\dots\dots}\underbrace{\dots\underbrace{0\dots00}_
{\min\{e,f\}}}
_{s-1}\\
(3b+2)_3&=\hskip-7pt\hphantom{\dots\dots}
\underbrace{\dots\underbrace{22\kern1pt.\kern1pt.\kern1pt222}_{
\min\{e,f\}+1}}_{s}
\end{align*}
The visualization of $(3b+2)_3$ has to be taken with a grain of salt
because, if $b=3^{s-2}-1$, then $(3b+2)_3$ has only $s-1$ digits.

For later use, we note that we must have $s-1>\min\{e,f\}$ and
\begin{equation} \label{eq:14}
b\ge 3^{s-2}-1+\chi(e=0).
\end{equation}
Moreover, if $s=2$, then $b+1$ has just one digit and therefore
necessarily $b=1$.
Another general observation is that
the number of carries when adding $3b$ and $2k-2b$ must be at least
\begin{equation} \label{eq:15}
e+\chi(f\ge e)\cdot(\chi(e>0)+\max\{0,f-s+1\}),
\end{equation}
because there must be carries when adding up the first, second, \dots, $e$-th digits (from the right), because if $f\ge e>0$ there must also
occur a carry when adding up the $(e+1)$-st digits,
and because $(2k+b+2)_3$ has 0's as
$s$-th, $(s+1)$-st, \dots, $f$-th digit. If we use \eqref{eq:14} and \eqref{eq:15} in \eqref{eq:13}, then we obtain the inequality
\begin{align} \notag
&v_3\(\frac {1} {2k+b+2}\binom {2k+b+2} {3b+2}\(\frac {2}
{3}\)^{k-b}\)\\
\notag
&\kern.4cm\ge3^{s-2}-1+\chi(e=0)-k-e+(e+\chi(f\ge
e)\cdot(\chi(e>0)+\max\{0,f-s+1\}))\\
&\kern.4cm\ge-k+3^{s-2}-1+\chi(f\ge e)+
\chi(f\ge e)\cdot\max\{0,f-s+1\}.
\label{eq:16}
\end{align}

Let first $k\not\equiv1$~(mod~3) (i.e., $h\not\equiv5$~(mod~9)) or, equivalently, $f=0$.
If we do the according simplifications in \eqref{eq:16}, then we arrive at
$$v_3\(\frac {1} {2k+b+2}\binom {2k+b+2} {3b+2}\(\frac {2}
{3}\)^{k-b}\)\ge-k+3^{s-2}-1\ge-k.$$
If $s>3$ the last inequality is in fact a strict inequality. If $s=2$ then necessarily $b=1$. In that case we could have used $b>0=3^{s-2}-1$ instead of \eqref{eq:14},
which would also lead to a strict inequality.
If we compare this with \eqref{eq:12}, then it follows that our
sum cannot vanish.

Now let $k\equiv1$~(mod~3) (i.e., $h\equiv5$~(mod~9)). Equivalently, $f\ge1$.
We combine the summands in \eqref{eq:11} for $b=0$ and $b=1$:
\begin{multline*}
\frac {1} {2k+2}\binom {2k+2}2\(\frac {2} {3}\)^k-
\frac {1} {2k+3}\binom {2k+3}5\(\frac {2} {3}\)^{k-1}\\
=
-\frac {2^{k-2}} {5\cdot3^k}{(2k+1)(2k^3+k^2-k-10)} .
\end{multline*}
The $3$-adic valuation of this expression is
$$v_3(2k+1)-k=-k+f.$$
We claim that we have
$$v_3\(\frac {1} {2k+b+2}\binom {2k+b+2} {3b+2}\(\frac {2}
{3}\)^{k-b}\)>-k+f$$
for all $b\ge2$. It should be noted that $b\ge2$ implies $s\ge3$.

\smallskip
(a) $f\ge s\ge3$. Since $s-1>\min\{e,f\}$, we must have $f>e$. Thus,
from \eqref{eq:16}, we get
\begin{align*}
v_3\(\frac {1} {2k+b+2}\binom {2k+b+2} {3b+2}\(\frac {2}
{3}\)^{k-b}\)
&\ge-k+3^{s-2}+(f-s+1)\\
&>-k+f,
\end{align*}
since $s\ge3$.

(b) $s>f\ge0$. From \eqref{eq:16}, we get
\begin{align*}
v_3\(\frac {1} {2k+b+2}\binom {2k+b+2} {3b+2}\(\frac {2}
{3}\)^{k-b}\)
&\ge-k+3^{s-2}-1\\
&\ge-k+\max\{1,3^{f-1}\}-1\\
&\ge-k+f,
\end{align*}
as long as $f\ge1$. If $f\ge3$, the inequality chain is in fact strict.
If $f=0$ or $f=1$ then, because of $s\ge3$, we could have used
$3^{s-2}>1$ to
obtain that the $3$-adic valuation in question must be at least
$-k+f=-k$. If $f=2$ and $s>3$, then we could have used the estimation
$3^{s-2}>3^{f-1}$ instead. The only remaining case is $f=2$ and
$s=3$. Since we must have $s-1>\min\{e,f\}$, ths only options for $e$
are $e=0$ or $e=1$.
If $e=0$ then, using \eqref{eq:14} and \eqref{eq:15} in \eqref{eq:13}, we obtain
$$
v_3\(\frac {1} {2k+b+2}\binom {2k+b+2} {3b+2}\(\frac {2}
{3}\)^{k-b}\)
=3^{s-2}-k>-k+2=-k+f.
$$
Finally, if $e=1$, then from the visualization of $(3b+2)_3$ we see
that there must be at least 2 carries when adding $3b+2$ and $2k-2b$
(namely when adding the first and second digits).
If we use this together with \eqref{eq:14} in \eqref{eq:13}, the we arrive at
$$
v_3\(\frac {1} {2k+b+2}\binom {2k+b+2} {3b+2}\(\frac {2}
{3}\)^{k-b}\)
=3^{s-2}-1-k-1+2>-k+2=-k+f.
$$

This completes the proof of our claim.

\end{document}